\documentclass[12pt, reqno]{amsart}
\usepackage{amsmath, amsthm, amscd, amsfonts, amssymb, graphicx, color}
\usepackage[bookmarksnumbered, colorlinks, plainpages]{hyperref}
\hypersetup{colorlinks=true,linkcolor=red, anchorcolor=green, citecolor=cyan, urlcolor=red, filecolor=magenta, pdftoolbar=true}

\newtheorem{theorem}{Theorem}[section]
\newtheorem{lemma}[theorem]{Lemma}
\newtheorem{prop}[theorem]{Proposition}
\newtheorem{cor}[theorem]{Corollary}
\theoremstyle{definition}

\theoremstyle{remark}
\newtheorem{remark}[theorem]{\bf{Remark}}
\numberwithin{equation}{section}
\begin{document}

\title [ Numerical radius inequalities of Hilbert space operators] {  \Small{ Furtherance of Numerical radius inequalities of Hilbert space operators }  }

\author[P. Bhunia and K. Paul]{Pintu Bhunia and Kallol Paul}

\address{(Bhunia) Department of Mathematics, Jadavpur University, Kolkata 700032, West Bengal, India}
\email{pintubhunia5206@gmail.com}

\address{(Paul) Department of Mathematics, Jadavpur University, Kolkata 700032, West Bengal, India}
\email{kalloldada@gmail.com;kallol.paul@jadavpuruniversity.in}

\thanks{First author would like to thank UGC, Govt. of India for the financial support in the form of SRF}
\thanks{}
\thanks{}



\maketitle

\begin{abstract}
If $A,B$ are bounded linear operators on a complex Hilbert space, then 
\begin{eqnarray*}
w(A) &\leq& \frac{1}{2}\left( \|A\|+\sqrt{r\left(|A||A^*|\right)}\right),\\
w(AB \pm BA)&\leq& 2\sqrt{2}\|B\|\sqrt{ w^2(A)-\frac{c^2(\Re (A))+c^2(\Im (A))}{2} },
\end{eqnarray*}
where $w(.),\|.\|,c(.)$ and $r(.)$ are the numerical radius, the operator norm, the Crawford number and the spectral radius respectively, and $\Re (A)$, $\Im (A)$ are the real part, the imaginary part of $A$ respectively. The inequalities obtained here generalize and improve on the existing well known inequalities.
\end{abstract}

\section{\textbf{Introduction}}

\noindent Let  $ \mathcal{H}$ be a complex Hilbert space with inner product  $\langle.,.\rangle$ and  let $ \mathcal{B}(\mathcal{H}) $ be the collection of all bounded linear operators on $ \mathcal{H}.$  As usual the norm induced by the inner product $\langle.,.\rangle$  is denoted by $ \Vert . \Vert.$  For  $A \in \mathcal{B}(\mathcal{H})$, let $\|A\|$ be the operator norm of $A,$ i.e., $ \Vert A \Vert = \sup_{\Vert x \Vert =1} \Vert Ax \Vert.$  For  $A \in \mathcal{B}(\mathcal{H})$, $A^*$ denotes the adjoint of $A$ and $|A|, |A^*|$  respectively denote the positive part of $A,A^*$, i.e., $|A|= (A^*A)^{\frac{1}{2}} , |A^*|= (AA^*)^{\frac{1}{2}}.$  Let $S_{\mathcal{H}}$ denote the unit sphere of the Hilbert space $\mathcal{H}.$  The numerical range of $A$, denoted by $W(A),$ is defined as  $W(A):=\big\{ \langle Ax,x\rangle~~:~~x\in S_{\mathcal{H}}\big\}.$ Considering the continuous mapping $ x \longmapsto \langle Ax,x \rangle $  from $S_{\mathcal{H}} $ to the scalar field $\mathbb{C},$  it is easy to see that $W(A)$ is a  compact subset of $\mathbb{C}$  if $\mathcal{H}$ is finite dimensional. The famous Toeplitz-Hausdorff theorem states that the numerical range is a convex set.
\noindent The numerical radius and the Crawford number of $A$, denoted as $w(A)$ and $c(A)$, respectively, are defined as 
\begin{eqnarray*}
w(A):= \sup_{x\in S_{\mathcal{H}}} |\langle Ax,x\rangle|
\end{eqnarray*}
and 
\begin{eqnarray*}
c(A):= \inf_{x\in S_{\mathcal{H}}} |\langle Ax,x\rangle|.
\end{eqnarray*}
The numerical radius is a norm on $\mathcal{B}(\mathcal{H}) $ satisfying the following inequality 
\begin{eqnarray}\label{eqv}
\frac{1}{2}\|A\|\leq w(A)\leq \|A\|.
\end{eqnarray}
Clearly, (\ref{eqv}) implies that the numerical radius norm is equivalent to the operator norm. The inequality  (\ref{eqv}) is sharp,  $w(A) = \|A\|$ if $AA^*=A^*A$ and $ w(A)=\frac{\|A\|}{2} $ if $A^2=0.$ For further readings on the  numerical range and the numerical radius of bounded linear operators, we refer to the  book  \cite{GR}.
The spectral radius of $A$, denoted as $r(A),$ is defined as 
\begin{eqnarray*}
r(A):= \sup_{\lambda \in \sigma(A)} |\lambda|,
\end{eqnarray*}
where $\sigma(A)$ is the spectrum of $A$. Since $\sigma(A)\subseteq \overline{W(A)}$, $r(A)\leq w(A)$. Also, $r(A)=w(A)$ if $A^*A=AA^*.$
Kittaneh \cite[Th. 1]{k05} and \cite[Th. 1]{k03} improved on the inequality (\ref{eqv}), to prove that 
\begin{eqnarray}\label{imp1}
\frac{1}{4}\left \|A^*A+AA^*\right \| &\leq&  w^2(A) \leq \frac{1}{2} \left \|A^*A+AA^*\right \|
\end{eqnarray}
and
\begin{eqnarray}\label{imp2}
w(A) &\leq& \frac{1}{2}\left(\|A\|+\sqrt{\|A^2\|}\right),
\end{eqnarray}
respectively.
Bhunia and Paul \cite[Cor. 2.5]{BP2} improved on the right hand inequalities of both  (\ref{eqv}) and (\ref{imp1}) to prove that 
\begin{eqnarray}\label{bound1comb}
	w^2(A)& \leq &  \min_{0\leq \alpha \leq 1} \left \| \alpha |A|^2 +(1-\alpha)|A^*|^2 \right \|.
\end{eqnarray} 
In \cite[Th. 2.1]{BP1}, Bhunia and Paul also improved on the left hand inequalities of both  (\ref{eqv}) and (\ref{imp1}) to prove that 
\begin{eqnarray*}
 \frac{1}{4}\|A^*A+AA^*\| &\leq&   \frac{1}{8}\big( \|A+A^*\|^2+\|A-A^*\|^2\big) \\
&  \leq & \frac{1}{8}\big( \|A+A^*\|^2+\|A-A^*\|^2\big) +\frac{1}{8}c^2\big(A+A^*\big)+\frac{1}{8}c^2\big(A-A^*\big)\\
&\leq&  w^2(A).
\end{eqnarray*}
Fong and Holbrook \cite{FH} obtained the remarkable numerical radius inequality that 
\begin{eqnarray}\label{Fong}
w(AB+ BA) \leq  2\sqrt{2} \|B\| w(A).
\end{eqnarray}
Hirzallah and Kittaneh \cite{HK} improved on the inequality (\ref{Fong}) in the following form:
\begin{eqnarray}
w(AB \pm BA)&\leq& 2\sqrt{2}\|B\|\sqrt{ w^2(A)-\frac{|~~\|\Re (A)\|^2-\|\Im (A)\|^2~~|}{2} }.
\end{eqnarray}

\noindent Over the years many mathematicians have developed various  inequalities  improving (\ref{eqv}), we refer to \cite{OK, BBP1,BBP2,BPN,BBP3,BBP4,BBP5} and references therein.\\
In this paper, we obtain an improvement and generalization of the inequality (\ref{imp2}). Some inequalities for the numerical radius of the commutators of bounded linear operators are also obtained, which improve on (\ref{Fong}).  
 
\section{\textbf{Improvement of inequality (\ref{imp2})}}

\noindent Our improvement of the inequality (\ref{imp2}), is stated as the following theorem:

\begin{theorem}\label{th-oprt1}
Let $A\in \mathcal{B}(\mathcal{H}).$ Then, $w(A) \leq \frac{1}{2}\left( \|A\|+\sqrt{r\left(|A||A^*|\right)}\right).$
\end{theorem}

\begin{remark}\label{rem}
	
	If $A\in \mathcal{B}(\mathcal{H})$, then $r\left(|A||A^*|\right)\leq w\left(|A||A^*|\right)\leq \|\left(|A||A^*|\right)\|=\|A^2\|.$ Hence, Theorem \ref{th-oprt1} improves (\ref{imp2}). To show proper improvement we consider 
	$A=\left(\begin{array}{ccc}
	1 & 4\\
	1 & 1
	\end{array}\right)$. Then $|A|=\left(\begin{array}{ccc}
	1 & 1\\
	1 & 4
	\end{array}\right)$ and  $|A^*|=\left(\begin{array}{ccc}
	4 & 1\\
	1 & 1
	\end{array}\right)$. It is easy to see that $r\left(|A||A^*|\right)=9<\|\left(|A||A^*|\right)\|=\|A^2\|=\sqrt{59+10\sqrt{34}}\approx 10.83.$

\end{remark}

In order to prove Theorem \ref{th-oprt1} we need the following sequence of lemmas. First lemma can be found in \cite{K}.

\begin{lemma}$($\cite[Cor. 2]{K}$)$\label{lem-positive1}
Let $A,B\in \mathcal{B}(\mathcal{H})$ be positive operators. Then
\[\|A+B\|\leq \max\{\|A\|, \|B\|  \}+\left\|A^{1/2}B^{1/2}\right\|.\]
\end{lemma}

The second lemma which contains a mixed schwarz inequality, can be found in \cite[pp. 75-76]{halmos}.

\begin{lemma}$($\cite[pp. 75-76]{halmos}$)$\label{lem-th-oprt}
Let $A\in \mathcal{B}(\mathcal{H})$. Then 
\[|\langle Ax,x\rangle|\leq \langle |A|x,x\rangle^{1/2}~~\langle |A^*|x,x\rangle^{1/2},~~\forall~~x\in \mathcal{H}.\]
\end{lemma}

The third lemma is as follows.

\begin{lemma}\label{lem-positive2}
Let $A,B\in \mathcal{B}(\mathcal{H})$ be positive operators. Then $$r(AB)=\left\|A^{1/2}B^{1/2}\right\|^2.$$
\end{lemma}

\begin{proof}
By commutativity property of the spectral radius we have that 
\begin{eqnarray*}
r(AB)&=&r\left(A^{1/2}A^{1/2}B^{1/2}B^{1/2}\right)=r\left(A^{1/2}B^{1/2}B^{1/2}A^{1/2}\right)\\
&=&r\left( A^{1/2}B^{1/2} \left(A^{1/2}B^{1/2}\right)^* \right)=\left\| A^{1/2}B^{1/2} \left(A^{1/2}B^{1/2}\right)^* \right\|\\
&=&\left\|A^{1/2}B^{1/2}\right\|^2.
\end{eqnarray*}
\end{proof}
Now we prove Theorem \ref{th-oprt1}.\\

\noindent \textbf{Proof of Theorem \ref{th-oprt1}.} Let $x\in S_{\mathcal{H}}.$ Then by Lemma \ref{lem-th-oprt} we get,
\begin{eqnarray*}
|\langle Ax,x\rangle|&\leq& \langle |A|x,x\rangle^{1/2}~~\langle |A^*|x,x\rangle^{1/2}\\
&\leq&  \frac{1}{2}(\langle |A|x,x\rangle+\langle |A^*|x,x\rangle)\\
&\leq& \frac{1}{2} \left\|~~ |A|+|A^*| ~~\right\|\\
&\leq& \frac{1}{2} \left(\|A\|+ \left\||A|^{1/2}|A^*|^{1/2}\right\|   \right), ~~\textit{by Lemma \ref{lem-positive1}}\\
&=& \frac{1}{2} \left(\|A\|+ \sqrt{r\left(|A||A^*|\right)}  \right), ~~\textit{by Lemma \ref{lem-positive2}}.
\end{eqnarray*}
Hence, by taking supremum over $x\in S_{\mathcal{H}}$ we get,
\begin{eqnarray*}
w(A) &\leq& \frac{1}{2}\left( \|A\|+\sqrt{r\left(|A||A^*|\right)}\right),
\end{eqnarray*}
This completes the proof.

As an application of Theorem \ref{th-oprt1}, we prove the following corollary.

\begin{cor}\label{corollary1}
	Let $A\in \mathcal{B}(\mathcal{H}).$ If $r(|A||A^*|)=0$,  then $w(A)=\frac{\|A\|}{2}.$  
\end{cor}

\begin{proof}
	It follows from (\ref{eqv}) and Theorem \ref{th-oprt1} that $$\frac{\|A\|}{2}\leq w(A)\leq \frac{1}{2}\left( \|A\|+\sqrt{r\left(|A||A^*|\right)}\right).$$ This implies that if $r(|A||A^*|)=0$,  then $w(A)=\frac{\|A\|}{2}.$  
\end{proof}

\begin{remark}\label{ex}
	It should be mentioned here that the converse of Corollary \ref{corollary1} does not hold if $\dim (\mathcal{H})\geq 3.$ As for example, we consider 
	$A=\left(\begin{array}{ccc}
		0 & 3 & 0\\
		0 & 0 & 0\\
		0 & 0 & 1
	\end{array}\right)$. Then we see that $w(A)=\frac{3}{2}=\frac{\|A\|}{2}$, but $r(|A||A^*|)\neq 0.$
\end{remark}

The following corollary is an immediate consequnece of Theorem \ref{th-oprt1}.

\begin{cor}\label{corollary2}
	Let $A\in \mathcal{B}(\mathcal{H}).$ If $w(A)=\frac{1}{2}\left(\|A\|+\sqrt{\|A^2\|}\right)$, then $r(|A||A^*|)=\|A^2\|.$
	
\end{cor}

\begin{proof}
	Using Remark \ref{rem}, it follows from  Theorem \ref{th-oprt1} that 
	\begin{eqnarray*}
		w(A) \leq \frac{1}{2}\left( \|A\|+\sqrt{r\left(|A||A^*|\right)}\right)\leq \frac{1}{2}\left(\|A\|+\sqrt{\|A^2\|}\right).
	\end{eqnarray*}
	This implies that if $w(A)=\frac{1}{2}\left(\|A\|+\sqrt{\|A^2\|}\right)$, then $r(|A||A^*|)=\|A^2\|.$
\end{proof}

\begin{remark}
	It should be mentioned that the converse of Corollary \ref{corollary2} is not true. Considering the same example as in Remark \ref{ex}, i.e., $A=\left(\begin{array}{ccc}
		0 & 3 & 0\\
		0 & 0 & 0\\
		0 & 0 & 1
	\end{array}\right)$, we see that  $r(|A||A^*|)=\|A^2\|=1,$ but $w(A)=\frac{3}{2}< 2=\frac{1}{2}\left(\|A\|+\sqrt{\|A^2\|}\right)$.
\end{remark}

We give a sufficient condition for $w(A) = \frac{1}{2}\left( \|A\|+\sqrt{r\left(|A||A^*|\right)}\right)$, when $A$ is a complex $n\times n$ matrix.

\begin{prop}\label{prop1}
	Let $A$ be a complex $n\times n$ matrix. Suppose $A$ satisfies either one of the following conditions.
	
	\noindent $(i)$ $A$ is unitarily similar to $[\alpha ] \oplus B$, where $B$ is an $(n-1) \times (n-1)$ matrix with $\|B\|\leq |\alpha|.$ 
	
	\noindent $(ii)$ $r(|A||A^*|)=0.$\\
	\noindent Then,  $w(A) = \frac{1}{2}\left( \|A\|+\sqrt{r\left(|A||A^*|\right)}\right)$.
	
\end{prop}

\begin{proof}
	Let $(i)$ holds. Then $w(A)=|\alpha|$ and $\|A\|=|\alpha|$. Also it is not difficult to verify that $r(|A||A^*|)=|\alpha|^2.$ Hence, $\frac{1}{2}\left( \|A\|+\sqrt{r\left(|A||A^*|\right)}\right)=|\alpha|$. Now let $(ii)$ holds. Then from Corollary \ref{corollary1} we get, $w(A) = \frac{1}{2}\left( \|A\|+\sqrt{r\left(|A||A^*|\right)}\right)=\frac{\|A\|}{2}$. Thus, we complete the proof.
	
\end{proof}

Next we give a generalized result of Theorem \ref{th-oprt1}. For this purpose we need the following lemma, which is the generalization of Lemma \ref{lem-th-oprt}.

\begin{lemma}$(${\cite[Th. 5]{K88}}$)$.\label{lemma-3g}
Let $A,B\in \mathcal{B}(\mathcal{H})$ be such that $|A|B=B^*|A|$ and let $f,g$ be non-negative continuous functions on $[0,\infty]$ satisfy $f(t)g(t)=t$, $\forall t\geq 0.$ Then, $|\langle ABx,y\rangle|\leq r(B) \|f(|A|)x\| \|g(|A^*|)y\|,~~\forall x,y \in \mathcal{H}. $
\end{lemma}

Using Lemma \ref{lemma-3g} and proceeding similarly  as in Theorem \ref{th-oprt1}, we can prove the following theorem.

\begin{theorem}\label{th-generalize}
Let $A,B\in \mathcal{B}(\mathcal{H})$  be such that $|A|B=B^*|A|$ and let $f$, $g$ be as in Lemma \ref{lemma-3g}. Then
\[w(AB)\leq \frac{r(B)}{2}\Big(\max \left\{ \|f(|A|)\|^2, \|g(|A^*|)\|^2 \right\} + \left\|~~  |f(|A|)|~~ |g(|A^*|)| ~~\right\|  \Big).\]   
\end{theorem}

Considering $f(t)=g(t)=\sqrt{t}$\,\, in Theorem \ref{th-generalize} we get the following corollary.

\begin{cor}\label{product}
Let $A,B\in \mathcal{B}(\mathcal{H})$  be such that $|A|B=B^*|A|$. Then
\begin{eqnarray*}
w(AB)&\leq& \frac{r(B)}{2}\left( \|A\|+ \sqrt{r\left(|A||A^*|\right)} \right)\\
&\leq& \frac{1}{4}\left( \|B\|+ \sqrt{r\left(|B||B^*|\right)} \right)\left( \|A\|+ \sqrt{r\left(|A||A^*|\right)} \right).
\end{eqnarray*}
\end{cor}

\begin{remark}
 If $A,B\in \mathcal{B}(\mathcal{H})$  be such that $|A|B=B^*|A|$, then Alomari \cite[Cor. 3.2]{A} proved that
\begin{eqnarray}\label{alomari}
w(AB)\leq \frac{1}{4}\left(\|B\|+\sqrt{\|B^2\|} \right) \left(\|A\|+\sqrt{\|A^2\|} \right). 
\end{eqnarray}
Clearly our inequalities in Corollary \ref{product} improve on the inequality (\ref{alomari}). 

\end{remark}

\section{\textbf{Improvement of inequality (\ref{Fong})}}

\noindent In order to obtain an improvement of the inequality (\ref{Fong}) we need the following lemma \cite{BP1} . First, we note the Cartesian decomposition of $A\in \mathcal{B}(\mathcal{H})$, i.e., $A=\Re (A)+{\rm i} \Im (A)$, where $\Re (A)=\frac{A+A^*}{2}$ and $\Im (A)=\frac{A-A^*}{2{\rm i}}$.

\begin{lemma} $($\cite[Cor. 2.3]{BP1}$)$\label{lem1} 
Let $A\in \mathcal{B}(\mathcal{H})$. Then 
\[\|AA^*+A^*A\|\leq 4\left[ w^2(A)-\frac{c^2(\Re (A))+c^2(\Im (A))}{2} \right].\]

\end{lemma}

Now we prove the desired result.

\begin{theorem}\label{th1}
Let $A,B,X,Y\in \mathcal{B}(\mathcal{H})$. Then 
\[w(AXB \pm BYA)\leq 2\sqrt{2}\|B\|\max  \left\{\|X\|,\|Y\| \right\}\sqrt{ w^2(A)-\frac{c^2(\Re (A))+c^2(\Im (A))}{2}   }.\]

\end{theorem}

\begin{proof}
First we assume that $\|X\|\leq 1$ and $\|Y\|\leq 1$. Let $x\in S_{\mathcal{H}}$. Then we have
\begin{eqnarray*}
|\langle (AX\pm YA)x,x\rangle|&\leq& |\langle AXx,x\rangle|+|\langle YAx,x\rangle| \\
&=& |\langle Xx,A^*x\rangle|+|\langle Ax,Y^*x\rangle| \\
&\leq& \|A^*x\|+ \|Ax\|, ~~\textit{by Cauchy Schwarz inequality} \\
&\leq &  \sqrt{2(\|A^*x\|^2+ \|Ax\|^2)},~~\textit{by convexity of $ f(x)=x^2$}\\
&\leq&  \sqrt{2\|AA^*+A^*A\|}\\
&\leq& 2\sqrt{2}\sqrt{ w^2(A)-\frac{c^2(\Re (A))+c^2(\Im (A))}{2}   }, ~~\textit{by Lemma \ref{lem1}}.
\end{eqnarray*}
Hence, by taking supremum over $\|x\|=1$ we get,
\begin{eqnarray}\label{eqnth1}
w(AX\pm YA)&\leq& 2\sqrt{2}\sqrt{ w^2(A)-\frac{c^2(\Re (A))+c^2(\Im (A))}{2} }.
\end{eqnarray}
Now we consider the general case, i.e., $X,Y\in \mathcal{B}(\mathcal{H})$ be arbitrary operators. If $X=Y=0$ then Theorem \ref{th1} holds trivially. Let $\max  \left\{\|X\|,\|Y\| \right\}\neq 0.$ Then clearly $\left \| \frac{X}{\max  \left\{\|X\|,\|Y\| \right\}}\right\|\leq 1$ and $\left \| \frac{Y}{\max  \left\{\|X\|,\|Y\| \right\}}\right\|\leq 1$. So,  replacing $X$ and $Y$ by $\frac{X}{\max  \left\{\|X\|,\|Y\| \right\}}$ and $\frac{Y}{\max  \left\{\|X\|,\|Y\| \right\}}$, respectively, in (\ref{eqnth1}) we get,
\begin{eqnarray}\label{eqnth2}
w(AX\pm YA)\leq 2\sqrt{2}\max  \left\{\|X\|,\|Y\| \right\}\sqrt{ w^2(A)-\frac{c^2(\Re (A))+c^2(\Im (A))}{2} }.
\end{eqnarray}
Now replacing $X$ by $XB$ and $Y$ by $BY$ in (\ref{eqnth2}) we get,
\begin{eqnarray*}
w(AXB\pm BYA)\leq 2\sqrt{2}\max  \left\{\|XB\|,\|BY\| \right\}\sqrt{ w^2(A)-\frac{c^2(\Re (A))+c^2(\Im (A))}{2} },
\end{eqnarray*}
which implies that
\begin{eqnarray*}
w(AXB\pm BYA) \leq 2\sqrt{2} \|B\| \max  \left\{\|X\|,\|Y\| \right\}\sqrt{ w^2(A)-\frac{c^2(\Re (A))+c^2(\Im (A))}{2} }.
\end{eqnarray*}
\end{proof}

On the basis of Theorem \ref{th1} we prove the following corollary. 

\begin{cor}\label{corth1}
Let $A,B\in \mathcal{B}(\mathcal{H})$. Then
\begin{eqnarray}\label{eqncor1}
w(AB\pm BA) &\leq & 2\sqrt{2} \|B\| \sqrt{ w^2(A)-\frac{c^2(\Re (A))+c^2(\Im (A))}{2} }.
\end{eqnarray}
and 
\begin{eqnarray}\label{eqncor2}
w(AB\pm BA) &\leq & 2\sqrt{2} \|A\| \sqrt{ w^2(B)-\frac{c^2(\Re (B))+c^2(\Im (B))}{2} }.
\end{eqnarray}
\end{cor}

\begin{proof}
By considering $X=Y=I$ in Theorem \ref{th1} we get, (\ref{eqncor1}). Interchanging $A$ and $B$ in (\ref{eqncor1}) we get, (\ref{eqncor2}). 

\end{proof}

\begin{remark}
 Clearly, the inequality (\ref{eqncor1}) is  stronger than the inequality (\ref{Fong}).

\end{remark}

As an application of the inequality (\ref{eqncor1}) we prove the following result.

\begin{cor}\label{corcor1}
Let $A,B\in \mathcal{B}(\mathcal{H})$ and let $B\neq 0$. 
If $w(AB\pm BA) = 2\sqrt{2} \|B\|w(A)$, then $0\in \overline{W(\Re (A))}\cap \overline{W(\Im (A))}$.
\end{cor}
\begin{proof}
Let $w(AB\pm BA) = 2\sqrt{2} \|B\|w(A)$. Then it follows from (\ref{eqncor1}) that  \[ w(A)=\sqrt{ w^2(A)-\frac{c^2(\Re (A))+c^2(\Im (A))}{2} }.\] Hence, $c^2(\Re (A))+c^2(\Im (A))=0,$ i.e., $c(\Re(A))=c(\Im (A))=0$. Therefore, there exist norm one sequences $\{x_n\}$ and $\{y_n\}$ in $\mathcal{H}$  such that $|\langle \Re(A)x_n,x_n\rangle| \to 0$ and $|\langle \Im(A)y_n,y_n\rangle| \to 0$ as $n\to \infty.$ So, $0\in \overline{W(\Re (A))}\cap \overline{W(\Im (A))}$.
\end{proof}
For our next result we need the following three lemmas, the first two of which can be found in \cite{OK} and \cite{HD}, respectively.
\begin{lemma} $($\cite[Remark 2.2]{OK}$)$\label{lem4}
Let $A,B,X,Y\in \mathcal{B}(\mathcal{H}).$ Then 
\[w^2(AX\pm BY)\leq \|AA^*+Y^*Y\| ~~\|X^*X+BB^*\|.\]
\end{lemma}
\begin{lemma} $($\cite[Th. 1.1]{HD}$)$\label{lem2}
Let $A,B,X,Y\in \mathcal{B}(\mathcal{H}).$ Then 
\[\left \| \left(\begin{array}{cc}
A & X\\
Y & B
\end{array}\right) \right \| \leq \left\| \left(\begin{array}{cc}
\|A\| & \|X\|\\
\|Y\| & \|B\|
\end{array}\right) \right\|.\]
\end{lemma}

The next lemma is as follows.

\begin{lemma}\label{lem3}
Let $A,B\in \mathcal{B}(\mathcal{H})$. Then
$\|AA^*+B^*B\|\leq \mu(A,B), $
where 
\[ \mu(A,B) = \frac{1}{2}\left[\|A\|^2+\|B\|^2+\sqrt{\left(\|A\|^2-\|B\|^2\right)^2+4\|BA\|^2} \right].\]

\end{lemma}

\begin{proof}
 $AA^*+B^*B$ being a self-adjoint operator, we have
\begin{eqnarray*}
\|AA^*+B^*B\|&=&r(AA^*+B^*B)\\
&=& r\left(\begin{array}{cc}
AA^*+B^*B & 0\\
0 & 0
\end{array}\right)\\
&=&r\left(\left(\begin{array}{cc}
|A^*| & |B|\\
0 & 0
\end{array}\right)\left(\begin{array}{cc}
|A^*| & 0\\
|B| & 0
\end{array}\right)\right)\\
&=& r\left(\left(\begin{array}{cc}
|A^*| & 0\\
|B| & 0
\end{array}\right)\left(\begin{array}{cc}
|A^*| & |B|\\
0 & 0
\end{array}\right)\right), ~~r(XY)=r(YX)\\
&=& r\left(\begin{array}{cc}
|A^*|^2 & |A^*||B|\\
|B||A^*| & |B|^2
\end{array}\right)\\
&=& \left\|\left(\begin{array}{cc}
|A^*|^2 & |A^*||B|\\
|B||A^*| & |B|^2
\end{array}\right)\right\|\\
&\leq& \left\| \left(\begin{array}{cc}
\|A\|^2 & \||A^*||B|\|\\
\||B||A^*| \| & \|B\|^2
\end{array}\right)\right\|,~~\textit{by Lemma \ref{lem2} } \\
&=& \left\| \left(\begin{array}{cc}
\|A\|^2 & \|BA\| \\
\|BA\| & \|B\|^2
\end{array}\right) \right\| \\
&=& \frac{1}{2}\left[\|A\|^2+\|B\|^2+\sqrt{\left(\|A\|^2-\|B\|^2\right)^2+4\|BA\|^2} \right].
\end{eqnarray*}
Hence, \[\|AA^*+B^*B\|\leq \mu(A,B).\] 

\end{proof}

\begin{remark}
Notice that $\mu(A,B)\leq \max\{ \|A\|^2, \|B\|^2\}+\|BA\|$. In particular, if $A=B$ then $\mu(A,A)= \|A\|^2+\|A^2\|$. Hence, we have 
$\|AA^*+A^*A\|\leq \|A\|^2+\|A^2\|$.

\end{remark}

Now we are in a position to prove the following result.

\begin{theorem}\label{th2}
Let $A,B,X,Y\in \mathcal{B}(\mathcal{H}).$ Then 
\[w(AX \pm BY)\leq \sqrt{\mu(A,Y)~~\mu(B,X)}.\]
\end{theorem}

\begin{proof}
The proof follows from Lemma \ref{lem4} and Lemma \ref{lem3}.

\end{proof}

An application of Theorem \ref{th2} we get the following corollary.

\begin{cor}\label{cor3}
Let $A,B\in \mathcal{B}(\mathcal{H})$. Then 
\[w(AB\pm BA)\leq \sqrt{\left(\|A\|^2+\|A^2\|\right) \left(\|B\|^2+\|B^2\|\right)}.\]

 \end{cor}

\begin{remark}
Let $A,B\in \mathcal{B}(\mathcal{H})$ with $A^2=B^2=0.$ Then it follows from Corollary \ref{cor3} that $w(AB\pm BA)\leq \|A\| \|B\| < 2\sqrt{2}\|B\|w(A)=\sqrt{2}\|A\|\|B\|$.

\end{remark}

\bibliographystyle{amsplain}

\begin{thebibliography}{99}




\bibitem{OK} A. Abu-Omar and F. Kittaneh, Numerical radius inequalities for products and commutators of operators, Houston J. Math. 41(4) (2015)  1163-1173.

\bibitem{A} M.W. Alomari, Refinements of some numerical radius inequalities for Hilbert space operators, Linear Multilinear Algebra (2019). \url{https://doi.org/10.1080/03081087.2019.1624682}


\bibitem{BBP1} S. Bag, P. Bhunia  and K. Paul, Bounds of numerical radius of bounded linear operator using $t$-Aluthge transform, Math. Inequal. Appl.  23(3) (2020) 991-1004.	

\bibitem{BBP2} P. Bhunia and K. Paul,  Some  improvement of numerical radius inequalities of operators and operator matrices, Linear Multilinear Algebra (2020). \url{https://doi.org/10.1080/03081087.2020.1781037}

\bibitem{BPN} P. Bhunia, K. Paul and R.K. Nayak, Sharp inequalities for the numerical radius of Hilbert space operators and operator matrices, Math. Inequal. Appl. (2020), to appear. 

\bibitem{BBP3} P. Bhunia, S. Bag and K. Paul, Bounds for zeros of a polynomial using numerical radius of Hilbert space operators, Ann. Funct. Anal. 12, 21 (2021). \url{https://doi.org/10.1007/s43034-020-00107-4}

\bibitem{BBP4} P. Bhunia, S. Bag and K. Paul, Numerical radius inequalities of operator matrices with applications, Linear Multilinear Algebra (2019) \url{https://doi.org/10.1080/03081087.2019.1634673}

\bibitem{BBP5} P. Bhunia, S. Bag and K. Paul, Numerical radius inequalities and its applications in estimation of zeros of polynomials, Linear Algebra Appl. 573 (2019) 166-177.


\bibitem{BP1} P. Bhunia and K. Paul, Refinements of norm and numerical radius inequalities, (2020). arXiv:2010.12750 [math.FA]


\bibitem{BP2} P. Bhunia and K. Paul, Proper improvement of well-known numerical radius inequalities and their applications, (2020),	arXiv:2009.03206 [math.FA].

\bibitem{FH} C.-K. Fong and J.A.R. Holbrook, Unitarily invariant operator norms, Canad. J. Math. 35(1983) 274-299.

\bibitem{GR} K.E. Gustafson and D.K.M. Rao, Numerical Range, Springer, New York, 1997.

\bibitem{halmos} P.R. Halmos, A Hilbert space problems book, Springer Verlag, New York, 1982.

\bibitem{HK} O. Hirzallah and F. Kittaneh, Numerical radius inequalities for several operators, Math. Scand. 114(1) (2014) 110-119.


\bibitem{HD} J.-C. Hou and H.-K. Du, Norm inequalities of positive operator matrices. Integr Equ Oper Theory 22(1995) 281-294.


\bibitem{k05}  F. Kittaneh, Numerical radius inequalities for Hilbert space operators, Studia Math. 168(1) (2005), 73-80.

\bibitem{k03}  F. Kittaneh, Numerical radius inequality and an estimate for the numerical radius of the Frobenius companion matrix, Studia Math. 158(1) (2003), 11-17.


\bibitem{K}  F. Kittaneh, Norm inequalities for certain operator sums, J. Funct. Anal. 143 (1997), 337-348.


\bibitem{K88} F. Kittaneh, Notes on some inequalities for Hilbert space operators, Publ. RIMS Kyoto Univ. 24 (1988) 283-293.






\end{thebibliography}

\end{document}